\documentclass[11pt]{CGC2}
\usepackage[english]{babel}
\usepackage{verbatim}
\usepackage{psfrag}
\usepackage{graphicx}

\def\cC{\mathcal C}

\def\cL{\mathcal L}
\def\cM{\mathcal M}

\def\cP{\mathcal P}
\def\cR{\mathcal R}

\def\cV{\mathcal V}

\def\cX{\mathcal X}

\def\deg{\mbox{\rm deg}}

\def\min{{\rm min}}

\def\dim{\mbox{\rm dim}}

\newcommand{\fq}{{\mathbb F}_q}

\begin{document}

\begin{frontmatter}

\title{Generalized AG codes as evaluation codes}

{\author{Marco Calderini}}
{\tt{(marco.calderini@unitn.it) }}\\
{{Department of Mathematics,
 University of Trento, Italy }}

{\author{Massimiliano Sala}} {\tt{(maxsalacodes@gmail.com)}}\\
{{Department of Mathematics,
 University of Trento, Italy }}

\runauthor{M.~Calderini, M.~Sala}

\begin{abstract}
We extend the construction of GAG codes to the case of evaluation codes. We estimate the minimum distance of these extended evaluation codes and we describe the connection to the one-point GAG codes.
\end{abstract}

\end{frontmatter}

\section{Introduction}
\label{intro}

In 1999, Xing, Niederreiter and Lam proposed \cite{CGC-cd-art-niedlxingam99,CGC-cd-art-xingniedlam99} two constructions of linear codes based on algebraic curves using points of arbitrary degree. These generalize the construction of Algebraic Geometry (AG) codes introduced by Goppa \cite{CGC-cd-art-goppa81,CGC-cd-art-goppa82}. \"Ozbudak and Stichtenoth \cite{CGC-cd-art-ozbsti99} showed that there is essentially only one new construction, namely that of Generalized Algebraic Geometric (GAG) codes, and introduced the notion of designed minimum distance for GAG codes.

Until now several papers have studied GAG codes in an algebraic geometry way, see e.g. \cite{CGC-cd-art-heydtmann02}, \cite{CGC-cd-art-dinniexin00}, \cite{CGC-cd-art-calfai12}, \cite{CGC-alg-art-xingyeo07}.

H{\o}holdt, van Lint and Pellikaan \cite{CGC-cd-book-AG_HB}  founded the theory of order domains and of the order domain codes (or evaluation codes) to simplify the description of one-point AG codes. The minimum distance  of evaluation codes can be found by applying bound that relies only on some relatively simple theory \cite{CGC-cd-book-AG_HB}.

Affine-variety codes, introduced by Fitzgerald and Lax in \cite{CGC-cd-art-lax}, are particularly interesting for their parameters and for a new efficient decoding system \cite{CGC-cd-prep-manumaxchiara12}. \cite{CGC-cd-art-geil08} presents the AG codes as an example of affine-variety codes and their relation with evaluation codes. 

In this paper we  will extend the construction of affine-variety codes to introduce the GAG codes as a particular example of these family of codes. We extend, also, the construction of the evaluation codes and we analyze a particular case of the one-point GAG codes into the setting of these new codes. 
The remainder of this paper contains the following sections.
\begin{itemize}

\item[-] In Section \ref{sec:1} we recall definitions and theorems about the minimum distance for affine-variety codes, order domain codes and generalized algebraic geometric codes.
\item[-] In Section \ref{sec:6} we introduce two constructions of linear codes, the extended affine-variety codes and the extended order domain codes, and we estimate a lower bound on the minimum distance for these families of codes.

\item[-] In section \ref{sec:9} we analyze the relation between an extended order domain code and a GAG code constructed from a rational point and we compare the relevant bounds on the minimum distance of the code.
\end{itemize}

\section{Preliminaries}
\label{sec:1}

\subsection{Affine-variety codes}
\label{sec:2}

Let $I \subseteq \fq[X_1,\dots,X_m]$ be an ideal, we define
$$
I_q=I+\langle X_1^q-X_1,\dots,X_m^q-X_m\rangle
$$
$$
R_q=\fq[X_1,\dots,X_m]/I_q
$$

Let
 $$
 V=\{P_1,\dots,P_n\}=\mathcal{V}_{\fq}(I)=\mathcal{V}_{\overline{\FF}_q}(I_q)
 $$ 
 be the variety of $I$ over $\fq$. Here $\overline{\FF}$ means the algebraic closure of the field $\FF$.
 
 Define the evaluation map $ev:R_q\to \fq^n$, the $\fq$-linear map such that 
\begin{equation}\label{eq:ev}
 ev(f+I_q)=(f(P_1),\dots,f(P_n)).
\end{equation}
The evaluation map is a vector space isomorphism.

\begin{definition}
Let $L$ be an $\fq$- vector subspace of $R_q$. We define the affine variety code 
$$
C(I,L) = ev(L).
$$
\end{definition}

The notation of this subsection comes from \cite{CGC-cd-art-lax}, where also the code $C(I,L)^\perp$ is called an affine-variety code. In this paper we will not consider this type of codes.

\subsection{Order domain conditions}
\label{sec:3}

Let $J\subseteq \FF[X_1,\dots,X_m]$ be an ideal and let $\prec$ be a fixed monomial ordering. Denote by $\mathcal{M}(X_1,\dots , X_m)$ the set of all monomials in the variables $X_1,\dots, X_m$. The footprint of $J$ (or Hilbert staircase) with respect to $\prec$ is the set
$$
\Delta_\prec(J) = \{m \in \mathcal{M}(X_1,\dots , X_m) \,|\, m \mbox{ is not the leading monomial
of any polynomial in } J\}.
$$

\begin{definition}
Let $I\subseteq \FF[X_1,\dots , X_m]$ be an ideal. Let $\prec_w$ be a generalized weighted degree ordering, $w:\cM\to\NN_0^r$. Assume $I$ possesses a Gr\"obner basis $\mathcal G$ such that:
\begin{itemize}
\item[(i)] any $g\in\mathcal G$ has exactly two monomials of highest weight in its support.
\item[(ii)]  no two monomials in $\Delta_{\prec_w}(I)$ are of the same weight.
\end{itemize}
 Then we say that $(I,\prec_w)$ satisfies the order domain conditions.
\end{definition}

Let $L\subseteq R_q$ be a subspace. By using Gaussian elimination any basis of $L$ can be transformed into a basis of the following form.

\begin{definition}
Let $\prec$ be a fixed monomial ordering and $k=\dim(L)$. A basis $\{b_1 +I_q, \dots , b_{k} +I_q\}$ for $L$ such that $Supp(b_i)\subseteq\Delta_{\prec}(I_q)$ for $i = 1,\dots,k$ and $\mathrm{lm}(b_1) \prec \dots\prec \mathrm{lm}(b_{k})$ is said to be well-behaving with respect to $\prec$. Here $\mathrm{lm}(f) $ means the leading monomial of $f$.
\end{definition}

The sequence $(\mathrm{lm}(b_1),\dots , \mathrm{lm}(b_k))$ is the same for all choices of well-behaving basis of $L$. So we define the set
$$
\square_\prec(L)=\{\mathrm{lm}(b_1),\dots , \mathrm{lm}(b_k)\}.
$$

\begin{definition}
 Assume $I$ and $\prec_w$ satisfy the order domain conditions. Let $\Gamma = w(\Delta_{\prec_w} (I))\subseteq\NN_0^r$ and $\Delta=\Delta_{\prec_w} (I_q)$. For any $\lambda\in w(\Delta)$ we define
 $$
\sigma_\Delta(\lambda)=\sigma(\lambda) =  |\{\eta\in w(\Delta) \,|\, \eta-\lambda \in \Gamma\}|.
 $$
\end{definition}

\begin{theorem}[Th. 4.27 in \cite{CGC-cd-art-geil08}]\label{th:dist}
Assume $(I,\prec_w)$ satisfies the order domain condition and let $L$ subspace of $R_q$ with $\{b_1 +I_q, \dots , b_{\dim(L)} +I_q\}$ well-behaving basis. Then the minimum distance of $C(I,L)$ is at least
$$
\min\{\sigma(w(\alpha)) \,|\, \alpha \in \square_{\prec_w}(L)\}.
$$

\end{theorem}

\begin{remark}\label{rem:subset}
Assume that the pair $(I,\prec_w)$ satisfies the order domain conditions. Let $U \subseteq \mathcal{V}_{\fq} (I)$. Every finite set of points is a variety and therefore there exists polynomials $h_1,\dots , h_r$ such that the vanishing ideal of $U$ equals
$$
I_U =I+\langle h_1,\dots,h_r\rangle.
$$
The estimates of the minimum distances of $C(I,L)$ can be adapted if these codes are made by evaluating in $U$ rather than in the entire variety, but we need to replace $I_q$ with $I_U$.
\end{remark}

\subsection{Weight functions and order domains}
\label{sec:4}

The concept of a weight function was introduced by H{\o}holdt et al. in \cite{CGC-cd-book-AG_HB} to simplify the treatment of one-point geometric AG codes and to propose a generalization to objects of higher dimensions than curves.

Let $(R,\rho,\Gamma)$ be an order domain, where $\Gamma\subseteq \NN^r$ is a semigroup and $\rho:R\to \Gamma\cup\{-\infty\}$ is a weight function.

From \cite{CGC-alg-art-geipel02}[Th. 10.4] we know that every order domain with a finitely generated semigroup, $\Gamma$, can be constructed as a factor ring, $\FF[X_1,\dots,X_m]/I$. Therefore it can be described in the language of Gr\"obner basis theory.

\begin{definition}\label{def:mor}
Let $R$ be an $\fq$-algebra. A surjective map $\phi : R \to \fq^n$ is called a morphism of $\fq$-algebras if $\phi$ is $\fq$-linear and if 
$$
\phi(fg)=\phi(f)* \phi(g)
$$
for all $f, g \in R$. Here $*$ is the component-wise product.
\end{definition}

\begin{definition}\label{def:del}
Let $(R,\rho,\Gamma)$ be an order domain over $\fq$ and $\{f_\lambda\, |\,\rho(f_\lambda)=\lambda, \lambda\in\Gamma\}$ be a basis. Let $\phi : R\to \fq^n$ be a morphism as in Definition \ref{def:mor}. Define $\alpha(1) = 0$. For $i = 2,\dots,n$ define recursively $\alpha(i)$ to be the smallest element in $\Gamma$ that is greater than $\alpha(1),\dots , \alpha(i-1)$ and satisfies
$$
\phi(f_{\alpha(i)})\notin Span_{\fq}\{\phi(f_\lambda)\,|\,\lambda\prec_{\NN^r} \alpha(i)\}. 
$$
Write $\Delta(R, \rho, \phi) =\{\alpha(1),\dots , \alpha(n)\}$.
\end{definition}

\begin{definition}
Let $R$ be an order domain over $\fq$ and let $\phi$ be a morphism . Fix a basis $\{f_\lambda\, |\,\rho(f_\lambda)=\lambda, \lambda\in\Gamma\}$ and let $\Delta=\Delta(R, \rho, \phi)$. For $\lambda\in\Gamma$ and $\delta\in\NN$ consider the codes
$$
E(\lambda) = Span_{\fq}\{\phi(f_\eta)\,|\,\eta\preceq_{\NN^r}\lambda\}
$$
$$
\tilde{E}(\delta) = Span_{\fq}\{\phi(f_\eta)\,|\,\eta\in\Delta\mbox{ and }\sigma_\Delta(\eta)\geq\delta\}.
$$

\end{definition}

\begin{theorem}[Th. 2 in \cite{CGC-cd-art-geil09}]\label{th:boundord}
The minimum distance of $E(\lambda$) is at least
$$
\min\{\sigma_\Delta(\eta)\,|\,\eta\preceq_{\NN^r}\lambda\}
$$
and the minimum distance of $\tilde{E}(\delta)$ is at least $\delta$.

\end{theorem}
\subsection{GAG codes}
\label{sec:5}

Let $\cX$ be a projective, geometrically irreducible, non-singular algebraic curve defined over the finite field $\FF_{q}$. Let $g$ be the genus of $\cX$. Let $\Phi$ be the Frobenius map on $\cX$, namely the map sending a point $P$ with homogeneous coordinates $(a_0,\ldots,a_r)$ to the point $\Phi(P)$ with coordinates $(a_0^q,\ldots,a_r^q)$. 

Let $P$ be a point of $\cX$. Then $\deg(P)$ denotes the degree of $P$, namely  the least positive integer $n$ such that $P$ is $\FF_{q^n}$-rational, and the closed point of $P$ is the set $O_\Phi(P)=\{P,\Phi(P),\ldots, \Phi^{n-1}(P)\}$.

Let $\cX$ be a curve, let $P_{1},\ldots,P_{s}$ be points of $\cX$ such that for every $i\neq j$ the closed points $O_\Phi(P_i)$ and $O_\Phi(P_j)$ are disjoint. 
Let  $G$ be an $\fq$-rational divisor that has support disjoint from any closed point $O_\Phi(P_i)$.
Let $k_i:=\deg(P_i)$. For $i=1,\ldots, s$ let $\pi_i:\FF_{q^{k_i}}\to C_i$ be an $\fq$-linear isomorphism from   
the finite field $\FF_{q^{k_i}}$ onto a linear $[n_i,k_i,d_i]$ code $C_i\subseteq \fq^{n_i}$. 

\begin{definition}
Let $n=\sum_{i=1}^s n_i$, and consider the $\fq$-linear map
$$
\pi:\left\{ 
\begin{array}{ccc}
\cL(G) &\to & \fq^n\\
f&\mapsto & (\pi_1(f(P_1),\ldots,\pi_s(f(P_s)))
\end{array}
\right.
$$
The image of $\pi$ is a Generalized Algebraic Geometric code
$$
C(P_1,\ldots,P_s;G;C_1,\ldots,C_s)=\pi(\cL(G)).
$$
Here $\cL(G)$ denotes the Riemann-Roch space of $G$ over $\fq$.
\end{definition}

The designed minimum distance $\bar d$ of $C(P_1,\ldots,P_s;G;C_1,\ldots,C_s)$ is defined as follows (see \cite{CGC-cd-art-ozbsti99}):
let 
$$
X= \bigg\{ S\subseteq \{1,\ldots,s\}\mid \sum_{i\in S} k_i \leq \deg (G)\bigg\}.
$$
Then 
$$
\bar d:=\min\bigg\{ \sum_{i\notin S} d_i \mid S\in X \bigg\}
$$

\begin{proposition}[Prop. 4.1 in \cite{CGC-cd-art-ozbsti99}]\label{prop:XNL}
If $\,\sum_{i=1}^s k_i > \deg (G)$, then $C(P_1,\ldots,P_s;G;C_1,\ldots,C_s)$ is an $[n,k,d]$ code with parameters
$$
k=\dim(\cL(G))\ge \deg (G)+1-g\quad \text{and} \quad d\ge \bar d.
$$

\end{proposition}

Throughout this paper, the codes $C_i$ will be called the inner codes of the GAG code.

\begin{remark}\label{oss:dist}
If we construct the GAG code using $P_1,\dots,P_s$ points of which $h$ are $\fq$-rational, a divisor $G$ with $\deg(G)\leq h$ and inner codes having minimum distance all equals to $1$, then the designed minimum distance is equal to $s-\deg(G)$.
\end{remark}

\section{New construction of codes}
\label{sec:6}

For any $v \in \fq^n$, let $\mathrm{w}_H(v)=|\{i\,|\,v_i\neq 0\}|$.
\subsection{Extended Affine-variety codes}
\label{sec:7}

Let $(I,\prec_w)$ satisfying the order domain condition and let $\cP=\{P_1,\dots,P_h\}\subseteq\mathcal{V}_{\overline{\FF}_q}(I)$, with $\deg (P_i)=r_i$ for $i=1,\dots,h$. As in Remark \ref{rem:subset} there is an ideal $J\subseteq\fq[X_1,\dots,X_m]$ such that $\cP=\mathcal{V}_{\overline{\FF}_q}(I+J)$. Let $I+J=I_\cP$.

Let $L$ be a space over $\fq$ with well-behaving basis $B=\{b_1 +I_\cP, \dots , b_k +I_\cP\}$, and for $i=1,\dots,h$ let $\pi_i:\FF_{q^{r_i}}\to C_i$ be an $\fq$-linear isomorphism from the finite field $\FF_{q^{r_i}}$ onto the inner code $C_i$ over $\fq$ with parameters $[n_i,r_i,d_i]$. 

\begin{definition}
Let $n=\sum_{i=1}^h n_i$, $\cP=\{P_1,\dots,P_h\}$ and $\cC=\{C_1,\dots,C_h\}$. Consider the $\fq$-linear map, 
$$
\overline{ev}:\left\{ 
\begin{array}{ccc}
L &\to & \fq^n\\
f&\mapsto & (\pi_1(f(P_1),\ldots,\pi_h(f(P_h)))
\end{array}
\right.
$$
Then the extended affine-variety code is
$$
\overline{ev}(L)=C(I,L,\cP,\cC).
$$
\end{definition}

\begin{theorem}\label{th:gagav}
Let $\Delta=\Delta_{\prec_w}(I_\cP)$, then $C(I,L,\cP,\cC)$ has minimum distance at least
$$
\delta\hat{d},
$$

where $\delta=\min\{\sigma(w(\alpha)\,|\,\alpha \in \square(L)\}$ and $\hat{d}=\min\{d_1,\dots,d_h\}$.

\end{theorem}

\begin{proof}
Let $r=m.c.m.\{r_1,\dots,r_h\}$ and $B$ be a well-behaving basis for $L$. Consider 
$$
L'=Span_{\mathbb{F}_{q^r}}B
$$
and let $ev(L')\subseteq ({\mathbb{F}_{q^r}})^h$ (where $ev$ is as in (\ref{eq:ev})) be the affine variety code over ${\mathbb{F}_{q^r}}$ restricted at the points $P_1,\dots,P_h$. From Theorem \ref{th:dist}, the minimum distance of this code is at least $\delta$.

Note that $L\subseteq L'$, then for every non zero $c\in ev(L)$ we have $\mathrm{w}_H(c)\geq \delta$.

Let $\bar c \in C(I,L,\cP,\cC)\setminus\{0\}$, then $\bar c=(\pi_1(f(P_1),\ldots,\pi_h(f(P_h)))$ for some $f$. So let $S=\{i\,|\, f(P_i)\neq 0\}$, we have
$$
\mathrm{w}_H(c)=\sum_{i=1}^r\mathrm{w}_H(\pi_i(f(P_i)))=\sum_{i\in S}d_i\geq \delta\hat{d}.
$$
\end{proof}

\begin{remark}
We can estimate the minimum distance of the extended code $C(I,L,\cP,\cC)$ also if the order domain conditions are not satisfy. We can look at the number of one-way well-behaving pairs (see Def. 4.8 in \cite{CGC-cd-art-geil08}) as in Th. 4.9 in \cite{CGC-cd-art-geil08} . So we are able to obtain a bound similar to Theorem \ref{th:gagav}.
\end{remark}

\subsection{Extended Order Domain codes}
\label{sec:8}

Let $(R,\rho,\Gamma)$ be an order domain over $\fq$ and $B$ be a well-behaving basis for $R$. Consider $R'=Span_{\mathbb{F}_{q^r}}B$, then $(R',\rho,\Gamma)$ is an order domain over ${\mathbb{F}_{q^r}}$. Note that $R\subseteq R'$.

Now let $\phi:R'\to \mathbb{F}_{q^r}^h$ be a morphism $\phi=(\phi_1,\dots,\phi_h)$. For $i=1,\dots,h$ define $r_i=\min\{l\,|\,\phi_i(R)\subseteq\mathbb{F}_{q^l}\}$.

Let $\Delta=\Delta(R',\rho,\Gamma)$ be as in Definition \ref{def:del}. For $i=1,\dots,h$ let $\pi_i:\FF_{q^{r_i}}\to C_i$ be an $\fq$-linear isomorphism from the finite field $\FF_{q^{r_i}}$ onto the inner code $C_i$ over $\fq$ with parameters $[n_i,r_i,d_i]$.  
\begin{definition}
Let $\cC=\{C_1,\dots,C_h\}$ and $\cR=\{r_1,\dots,r_h\}$. For $\lambda\in\Gamma$ and $\delta\in\NN$ consider the codes
$$
{E}(\lambda,\cR,\cC) = Span_{\fq}\{(\pi_1(\phi_1(f_\eta)),\dots,\pi_h(\phi_h(f_\eta)))\,|\,\eta\preceq_{\NN^r}\lambda\}
$$
$$
\hat{E}(\delta,\cR,\cC) = Span_{\fq}\{(\pi_1(\phi_1(f_\eta)),\dots,\pi_h(\phi_h(f_\eta)))\,|\,\eta\in\Delta\mbox{ and }\sigma_\Delta(\eta)\geq\delta\}.
$$

\end{definition}

\begin{theorem}

The minimum distance of ${E}(\lambda,\cR,\cC)$ is at least
$$
\gamma\hat{d},
$$
where $\gamma=\min\{\sigma_\Delta(\eta)\,|\,\eta\preceq_{\NN^r}\lambda\}$ and $\hat{d}=\min\{d_1,\dots,d_h\}$.

The minimum distance of $\hat{E}(\delta,\cR,\cC)$ is at least $\delta \hat{d}$.
\end{theorem}
\begin{proof}
Obvious adaption of the proof at Theorem \ref{th:gagav}.
\end{proof}

\section{One-point GAG codes as Extended Order Domain codes}
\label{sec:9}

Now we consider the GAG codes constructed from a rational point of the curve, using as inner code $C_i=\mathbb{F}_q^{r_i}$ for $i=1,\dots,h$. We refer to these as one-point GAG codes.

Let $P$ be a rational point of a curve $\cX$ defined over a field $\fq$. Let $\nu_P$ be the valuation corresponding to $P$. Consider the algebraic structure
\begin{equation}\label{eq:alg}
R = \bigcup_{m=0}^\infty \cL(mP).
\end{equation}

Defining $\rho =-\nu_P$ we have $\rho(R) = \Gamma\cup\{-\infty\}$ where $\Gamma\subseteq \NN$ is known as the Weierstrass semigroup corresponding to $P$. By inspection $(R,\rho,\Gamma)$ is an order domain over $\fq$. 

Let $P_1,\dots,P_h$ be distinct points, with distinct closed points, of degree $r_1,\dots,r_h$, respectively.
Let $B$ be a well-behaving basis for $R$. Define $R'=Span_{\mathbb{F}_{q^r}}B$ and let $\phi:R'\to \mathbb{F}_{q^r}^h$ be a morphism with $\phi(f)=(f(P_1),\dots,f(P_h))$.
Then we have
$$
C(P_1,\dots,P_h,\lambda P,C_1,\dots,C_h)=C(I,L,\cP,\cC)=E(\lambda,\cR,\cC),
$$
where $L=\{f\,|\,\rho(f)\leq\lambda\}$, $\cP=\{P_1,\dots,P_h\}$, $\cR=\{r_1,\dots,r_h\}$ and $\cC=\{C_1,\dots,C_h\}$.

\begin{lemma}[Lemma 2 in \cite{CGC-cd-art-geil09}]
Let $\Gamma = \{\lambda_1,\lambda_2,\dots\}$ with $\lambda_1 <\lambda_2 <\dots$ be a numerical semigroup with finitely many gaps. For any $\lambda_i$ we have
$$
 \#(\Gamma\setminus(\lambda_i +\Gamma))=\lambda_i.
 $$
\end{lemma}

\begin{theorem}\label{th:pointgag}
The minimum distance of $E(\lambda,\cR,\cC)$ is at least
$$
\min\{\sigma_\Delta(\eta)\,|\,\eta\leq\lambda\}\geq h-\lambda
$$
where $\Delta=\Delta(R',\rho,\phi)$. 

\end{theorem}
\begin{proof}
The distances of the inner codes are all equal to $1$. Consider $\lambda_i\in\Delta$, with $\lambda_i\leq\lambda$. We have $\sigma(\lambda_i)=\#(\Delta\cap(\lambda_i+\Gamma))$, the elements in $\Delta$ that are not in $\lambda_i+\Gamma$ are at most $\lambda_i$. Then $\sigma(\lambda_i)\geq h-\lambda_i\geq h- \lambda$.
\end{proof}

\begin{remark}
With order domain code it is possible, sometimes, to have a bound on the minimum distance of a one-point Algebraic Geometry code better than the Goppa bound \cite{CGC-cd-art-geil09}. So also for GAG codes, if we are in the case as in the Remark \ref{oss:dist}, using the order domains is possible to obtain a bound always at least as good as (and sometimes better than) the bound in the Proposition \ref{prop:XNL}.
\end{remark}

\begin{example}
Let $\mathbb{F}_4=\{0,1,\alpha,\alpha^2\}$, where $\alpha$ is a primitive element. Consider the plane curve of affine equation $\cX:X^6+Y^5+Y$. Let $\prec$ be the weighted degree lexicographic ordering given by $w(X) = 5, w(Y ) = 6$. Let $I=\langle X^6+Y^5+Y\rangle$, then $(I,\prec)$ satisfies the order domain conditions  and $w(\Delta(I))$ is the semigroup $\langle 5,6\rangle$.

We have $8$ $\mathbb{F}_4$-rational points
$$
\cV(I_4)=\{(0,0),(0,1),(1,\alpha),(1,\alpha^2),(\alpha,\alpha),(\alpha,\alpha^2),(\alpha^2,\alpha),(\alpha^2,\alpha^2)\}
$$
and $\mathcal G=\{Y^2+X^3+Y,XY^2+XY+X,Y^4+Y\}$ is a Gr\"obner basis for $I_4$. The monomials in the footprint of $I_4$ are
$$
\Delta(I_4)=\{1,X,Y,X^2,XY,Y^2,X^2Y,Y^3\}
$$
and its corresponding weights are
$$
w(\Delta(I_4))=\{0,5,6,10,11,12,16,18\}.
$$

Now we consider a point of the variety $\mathcal{V}_{\overline{\mathbb{F}}_4}(I)$ of degree $3$ (there are not points of degree $2$). 
Let $\mathbb{F}_{64}={\mathbb{F}_4}[Z]/\langle Z^3+Z+1\rangle$ and let $\beta^3=\beta+1$. The point that we consider is $(1,\beta^3)$.
Using Buchberger-M\"oller's algorithm we can compute the Gr\"obner basis of the vanishing ideal of the nine points, so we adjoint the monomial $X^3$ at the footprint and the weight $15$ to $w(\Delta(I_4))$.

Consider now $L=Span_{\fq}\{1,X,Y\}$, then the minimum distance of $C(I,L,\cP,\cC)$, where the inner code used are $C_1=\dots=C_8=\mathbb{F}_4$ and $C_9=\mathbb{F}_{4}^3$, is at least $\min\{\sigma(0),\sigma(5),\sigma(6)\}=5$. This value improves on what obtainable from the GAG construction, as follows.

Looking at this code as a one-point GAG code we can note that the semigroup $w(\Delta(I))$ is the Weiestrass semigroup of the unique rational point at infinity, $P_\infty$, of the curve and $L=\cL(6 P_\infty)$. Therefore the bound on minimum distance of the GAG code as in Proposition \ref{prop:XNL} is equal to $3$.

\end{example}

In \cite{CGC-cd-art-matsumoto99} was shown that an order domain with numerical weight function (i.e. the weights are in $\NN_0$) is a sub algebra of a structure as in (\ref{eq:alg}). If the semigroup related to the order domain are not numerical then they are related to structures of transcendence degree greater than one, that is, these structures are  curves no longer (\cite{CGC-alg-art-geipel02} Sec. 11). Examples of evaluation codes coming from higher dimensional objects than curves are given in \cite{CGC-cd-art-AndeGeil} and these codes can be viewed as generalizations of one-point AG codes. Then our extension can be consider a generalization of the one-point GAG codes.

\bibliography{RefsCGC}

\end{document}